\documentclass[11 pt]{article} 

\usepackage{amssymb, amsmath, amsthm, bbm, mathrsfs, color,url}
\usepackage{fullpage}

\theoremstyle{definition}
\newtheorem{defn}{Definition}

\newtheorem{lem}[defn]{Lemma}
\newtheorem{ex}[defn]{Example}

\theoremstyle{plain}
\newtheorem{thm}[defn]{Theorem}

\newtheorem{conj}[defn]{Conjecture}

\newcommand{\bR}{\ensuremath{\mathbb{R}}}
\newcommand{\bC}{\ensuremath{\mathbb{C}}}
\newcommand{\bZ}{\ensuremath{\mathbb{Z}}}
\newcommand{\bF}{\ensuremath{\mathbb{F}}}
\newcommand{\bQ}{\ensuremath{\mathbb{Q}}}
\newcommand{\bN}{\ensuremath{\mathbb{N}}}
\newcommand{\cW}{\ensuremath{\mathcal{W}}}

\newcommand{\cA}{\ensuremath{\mathcal{A}}}

\newcommand{\beq}{\begin{equation}}
\newcommand{\eeq}{\end{equation}}
\newcommand{\absip}[2]{\ensuremath{\left\vert\langle {#1},{#2} \rangle \right\vert}}

\newcommand{\abs}[1]{\ensuremath{\left\vert{#1} \right\vert}}
\newcommand{\lspan}{\operatorname{span}}

\newcommand{\tr}{\operatorname{tr}}
\newcommand{\FF}{\operatorname{TFF}}

\begin{document}

\title{New Constructions and Characterizations of Flat and Almost Flat Grassmannian Fusion Frames}
\author{Emily J.~King}

\maketitle

\begin{abstract} Configurations of subspaces like equichordal and equiisoclinic tight fusion frames, which are in some sense optimally spread apart and which also have reconstruction properties emulating those of orthonormal bases, are useful in various applications, such as wireless communications and quantum information theory.  In this paper, a new construction of infinite classes of equichordal tight fusion frames built on semiregular divisible difference sets is presented.  Sometimes this construction yields an equiisoclinic packing. Each of the constructed fusion frames is shown to have both a flat representation and a sparse representation. Furthermore, integrality conditions which characterize when equichordal and equiisoclinic fusion frames can have orthonormal bases with entries in a subring of the algebraic integers are proven.  
 \end{abstract}

\section{Introduction}
It is often of interest in fields like coding theory (see, e.g., \cite{Cre08,PWTH16,XZG05,KaPa03,GrassFus}), quantum information theory  (see, e.g., \cite{FHS17,AFZ15,ShSl98,GoRo09}), and more to find configurations of subspaces which in some sense emulate orthonormal bases.  That is, the orthogonal projections onto the subspaces provide a resolution of the identity and the subspaces are as close to ``orthogonal'' as possible for a such a typically overcomplete system (i.e., form an optimal packing of the Grassmannian under chordal distance). One may also ask that the entries of some orthonormal bases for the subspaces come from a finite alphabet, like a set of roots of unity \cite{STDH07,Massey1993,KaPa03}.

Most non-trivial constructions of optimal configurations of $1$-dimensional subspaces (so-called Grassmannian frames) arise from combinatorial designs or group actions. (See~\cite{FM15,JKM19} and the references therein.)  There is also work using algebraic number theory~\cite{ACFW18}, quantifier elimination~\cite{MiPa19}, or intuition and luck~\cite{FJM18} to ``exactify'' numerically found packings.  (See also~\cite{CKM16} in which Newton-Kantorovich is used to prove the existence of certain optimal packings given ``sufficiently good'' numerical approximations.)  It is also possible to generate optimal subspace configurations (not just $1$-dimensional) by transforming another optimal configuration. (See~\cite{King19b,FMW21} and the references therein.) In this paper, a new construction of optimal subspace packings arising from semiregular divisible difference sets (Theorem~\ref{thm:construct}) will be presented. This construction will then be modified so that it may be sparsely represented (Theorem~\ref{thm:sparse}). Finally, basic results about algebraic integers will be leveraged to prove necessary conditions for these Grassmannian packings to be represented by orthonormal bases with constant or almost constant modulus, i.e., which are flat or almost flat (Theorems~\ref{thm:integral1} and~\ref{thm:inteqio}). 

The structure of the paper is as follows.  In Section~\ref{sec:frame}, we will define Grassmannian fusion frames, which are optimal subspace configurations that emulate orthonormal bases.  The existing constructions of optimal configurations of lines via difference sets will also be reviewed.  The new constructions appear in Section~\ref{sec:maincon} and the necessary conditions for (almost) flatness are in Section~\ref{sec:intc}.

Throughout the paper, $\bF$ will always either denote $\bR$ or $\bC$.  Further for $n \in \bN$, we define $[n] := \{1, 2, \hdots, n\}$.

\section{Mathematical Background}\label{sec:frame}
\subsection{Fusion Frames}
We consider the properties of orthonormal bases which make them useful in applications.  Namely, for an orthonormal basis $\{ e_i \}_{i=1}^k$ of $\bF^k$, the sum of the rank-$1$ projections onto the subspaces spanned by each of the vectors gives a resolution of the identity
\begin{equation}\label{eq:ONBres}
I_{k} = \sum_{i=1}^k e_i e_i^\ast;
\end{equation}
each of the $e_i$ are unit norm, ensuring that no vector is more important than the others; and the angles between the lines spanned by the vectors are as large as possible for a set of $k$ lines in $\bF^k$, meaning that each vector represents different information.  We further note that the components of the vectors of the Fourier basis are equal in modulus while the non-zero components of the standard basis are equal in modulus, easing implementation.

Grassmannian fusion frames are flexible systems which provide representations of data that are more robust to erasures than orthonormal bases but which emulate the above-listed properties of orthonormal bases. Fusion frames (not necessarily Grassmannian) have also been known in the literature as \emph{stable space splittings of Hilbert spaces}~\cite{Osw94,Osw97}, \emph{systems of bounded quasi-projectors}~\cite{Forn03,Forn04}, \emph{(weighted projective) resolutions of the identity}~\cite{Forn03,Bod07}, \emph{$g$-frames}~\cite{Sun06}, and \emph{frames of subspaces}~\cite{CaK04}. We will use the name \emph{fusion frames} and the related terminology, which were introduced in~\cite{CKL08}. See \cite[Chapter 13]{CaKBook} for a general overview of fusion frames.

\begin{defn} 
A finite collection of subspaces $\lbrace \mathcal{W}_i \rbrace_{i=1}^n$ in $\mathbb{F}^k$ is a \emph{tight (equidimensional) fusion frame (with unit weights)} for $\mathbb{F}^k$ if there exists an $A > 0$ (called the \emph{fusion frame bound}) satisfying
\begin{equation}\label{eqn:tight4}
x = \frac{1}{A} \sum_{i=1}^n P_i x, \quad \textrm{for all $x \in \bF^k$},
\end{equation}
where $P_i$ is an orthogonal projection onto $\mathcal{W}_i$. The map $x \mapsto \sum_{i=1}^n P_i x$ is called the \emph{fusion frame operator}.
\end{defn}
In this article, the subspaces will always be equal dimensional of dimension $m$. We will write $\lbrace \mathcal{W}_i \rbrace_{i} \in \FF(\bF^k,m,n)$ when $\lbrace \mathcal{W}_i \rbrace_{i}$ forms a tight fusion frame of $m$-dimensional subspaces  with unit weights for $\bF^k$.  See \cite{BaEh13,EhGr16,BLR15} for results about fusion frames and packings when the dimensions of the subspaces are unequal. When each $\cW_i$ is of dimension $m=1$ in a tight fusion frame with unit weights, the result is known as a \emph{finite unit norm tight frame (FUNTF)} (see, e.g., \cite{CaKBook,Waldron18}).  The resolution of the identity in~\eqref{eqn:tight4} is the analog of~\eqref{eq:ONBres}, and the fact that we are not weighting the projections in~\eqref{eqn:tight4} generalizes the unit-norm condition of orthonormal bases. We will use the chordal distance in order to expand the concept of orthogonality from bases of vectors to possibly redundant systems of subspaces.
\begin{defn}\label{defn:chor}
For $1 \leq m \leq k$, set $Gr(k,m)$ to be the collection of $m$ dimensional subspaces of $\mathbb{F}^k$.  $Gr(k,m)$ is called a \emph{Grassmannian}.  One metric that $Gr(k,m)$ may be endowed with is the \emph{chordal distance} (see, e.g., \cite{GrassPack})
\begin{equation*}
d_c(\mathcal{W}_i,\mathcal{W}_j) = [m - \tr(P_i P_j)]^{1/2},
\end{equation*}
for $\mathcal{W}_i, \mathcal{W}_j \in Gr(k,m)$, where $P_i$ is the orthogonal projection onto $\mathcal{W}_i$.  
\end{defn}
The \emph{Grassmannian packing problem} is the problem of finding $n$ elements in $Gr(k,m)$ so that the minimal distance between any two of them is as large as possible.  A numerical approach to solving this problem may be found in \cite{DHST08}, and a list of various packings is posted on \cite{Sloane}.  
\begin{defn}\label{defn:diffeq} \cite{LemSei73,BjGo73}
Let $\{ \cW_i\}_{i=1}^n \subset Gr(k,m)$ (not necessarily a fusion frame) with corresponding orthonormal bases as the columns of $\{L_i\}_{i=1}^n$. Then we say
\begin{itemize}
\item $\{ \cW_i\}_{i=1}^n$ is \emph{equichordal} when for all $i,j \in [n]$ with $i \neq j$, $\tr(L_i^\ast L_j L_j^\ast L_i)$ is constant; and
\item $\{ \cW_i\}_{i=1}^n$ is \emph{equiisoclinic} when there exists an $\alpha > 0$ such that for all $i,j \in [n]$ with $i \neq j$, $L_i^\ast L_j L_j^\ast L_i =\alpha I_m$.
\end{itemize}
\end{defn}
\begin{thm}\label{thm:simplex} \cite{Ran55,GrassPack}
Let $\{ \cW_i\}_{i=1}^n \subset Gr(k,m)$, then
\beq\label{eqn:simplex}
\min_{i,j \in [n], i\neq j}  d^2_c(\mathcal{W}_i,\mathcal{W}_j)  \leq \frac{m(k-m)n}{k(n-1)}.
\eeq
The bound in~\eqref{eqn:simplex} is saturated if and only if $\{ \cW_i\}_{i=1}^n$ is an \emph{equichordal tight fusion frame}.  For fixed parameters $n$, $m$, $k$, and $\bF$, the maximizers  $\{ \cW_i\}_{i=1}^n$ of~\eqref{eqn:simplex} are called \emph{Grassmannian fusion frames}.
\end{thm}

Thus if a tight fusion frame is equichordal or equiisoclinic it is a Grassmannian fusion frame since either of those configurations, when they exist, have equal and thus optimal pairwise chordal distances; however, for many parameter sets $n$, $m$, $k$, and $\bF$ there does not exist an equichordal tight fusion frame.  When $m=1$, an equichordal tight fusion frame is called an \emph{equiangular tight frame}.

Using basic trace arguments (see, e.g., \cite{FJMW17}), one can show that if  $\lbrace \mathcal{W}_i \rbrace_{i} \in \FF(\bF^k,m,n)$, then the fusion frame bound must be $A = \frac{nm}{k}$.  By similar arguments, one can show that if further $\lbrace \mathcal{W}_i \rbrace_{i}$ is equiisoclinic with \emph{equiisoclinic parameter} $\alpha$, then 
\beq\label{eqn:equiiso}
\alpha = \frac{mn-k}{k(n-1)}.
\eeq

 Under various models (e.g., deterministic or probabilistic, equichordal or equiisoclinic, etc.), Grassmannian fusion frames are optimally robust to noise and erasures~\cite{Bod07,CaK08,GrassFus,SAH14,EKB10,GKK01,StH03,HoPa04}. We end this section by introducing the concepts of flatness and almost flatness and fixing some final notation.
\begin{defn}
Given $\lbrace \mathcal{W}_i \rbrace_{i} \in \FF(\bF^k,m,n)$, we fix for each $i \in [n]$ an orthonormal basis $\{e_j^i\}_{j=1}^m$ for the subspace $\cW_i$ and denote by $L_i$ the $k \times m$ matrix $(e_1^i e_2^i \hdots e_m^i)$. We further define
\[
L  = \left( \begin{array}{cccc} L_1 & L_2 & \hdots & L_n\end{array}\right).
\]
If we can choose orthonormal bases for the subspaces such that the entries of $L$ all have the same modulus, then we say that the vectors and the associated fusion frame (with respect to the choice of orthonormal bases) are \emph{flat}, and similarly if all of the nonzero entries of $L$ have the same modulus, we say they are \emph{almost flat}.
\end{defn}
Flat fusion frames are the redundant, higher dimensional analogs of the Fourier bases, while almost flat fusion frames generalize the standard orthonormal bases.  Note that  the fusion frame operator is equal to $LL^\ast$ and for all $i$, $P_i =  L_i L_i^\ast$.


\subsection{Difference Sets}\label{sec:diff}
One class of constructions of equiangular tight frames uses difference sets and characters \cite{StH03,DiFe07,GoRo09,XZG05}.  For a general reference about difference sets and character theory, see~\cite{Pott}.
\begin{defn}
Let $G$ be a finite abelian group of size $n$. Define $\widehat{G}$ to be the collection of all homomorphisms $\chi: G \rightarrow S^1 \subset \bC$. Endowed with pointwise multiplication, $\widehat{G}$ forms a group which is isomorphic to $G$.  The elements are called \emph{characters} and the character $\chi_0$ which maps all elements of $G$ to $1$ is called the \emph{principal character}.
\end{defn}
For any subset $S \subseteq G$ and any character $\chi \in \widehat{G}$, we define the following element in $\bC$
\[
\chi(S) = \sum_{s \in S} \chi(s).
\]
Characters satisfy certain orthogonality relations.
\begin{lem}\label{lem:orthrel}
Let $G$ be a finite abelian group of size $n$. Then
\[
\chi(G)=  \left\{ \begin{array}{lr} n & \textrm{if $\chi = \chi_0$} \\ 0 & \textrm{if $\chi \neq \chi_0$}  \end{array} \right. \quad \textrm{and} \quad \sum_{\chi \in \widehat{G}} \chi(g)=  \left\{ \begin{array}{lr} n & \textrm{if $g=0$} \\ 0 & \textrm{if $g \neq 0$}  \end{array} \right..
\]
\end{lem}
The orthogonality relations are important to us because of their application to character tables.
\begin{defn}
Let $G$ be a finite abelian group of size $n$. Let $X_G$ be an $n \times n$ matrix with rows labeled by elements of $G$ and columns by elements of $\widehat{G}$.  Define the element in row $g$ column $\chi$ to be $\chi(g)$.  We call $X_G$ the \emph{character table} of $G$.
\end{defn}
By construction, each element of $X_G$ has modulus $1$.  Furthermore, it follows from Lemma~\ref{lem:orthrel} that $(1/\sqrt{n}) X_G$ is an orthogonal matrix.
\begin{ex}\label{ex:dft}
Let $G = \bZ_n$, the integers mod $n$.  Then the $n \times n$ (properly scaled) discrete Fourier transform matrix $(e^{-2\pi i jk/n})_{j,k=0}^n$ is $X_G$. If $G = \oplus_{\ell=0}^L \bZ_{n_\ell}$, then $X_G$ is the Kronecker product of the $X_{\bZ_{n_\ell}}$, ordered lexicograhically.
\end{ex}
Since the rows of character tables are equal norm and orthogonal, if we take any subset $S \subset G$ and choose a submatrix of $X_G$ with rows denoted by $S$ and columns all of $\widehat{G}$, the columns of the submatrix $L$ will always form a flat finite unit norm tight frame after appropriate scaling since $L L^\ast$ is the frame operator.  In order to generate an equiangular tight frame we must be more careful about how we select the rows.
\begin{defn}\label{defn:ds}
Let $G$ be a finite abelian group of size $n$.  If $D \subseteq G$ is a subset of size $k$ such that the multiset 
\[
\Delta(D) = \{ d_i - d_j: d_i, d_j \in D, d_i \neq d_j\}
\]
contains $\lambda$ copies of each non-identity element of $G$, then we say that $D$ is an $(n,k,\lambda)$-difference set.
\end{defn}
Examples of constructions of equiangular tight frames using difference sets in $G = \bZ_n$ or $\oplus_{i=1}^r \bZ_2$ appeared in \cite{StH03,DiFe07,GoRo09}, while the following theorem may be found in \cite{XZG05}. 
\begin{thm}\label{thm:DS1}
For an abelian group $G$ of size $n$, let $D \subseteq G$ have size $k$. We will write the elements of $D$ without indices, but set $\{\chi_i\}_{i=1}^n$ as an enumeration of $\widehat{G}$.  For each $1 \leq i \leq n$, we define the vector
\[
e_i : = \frac{1}{\sqrt{k}}\left( \chi_i(g)\right)_{g \in D} \in \bF^k.
\]
Then $\{e_i\}_{i=1}^n$ is an equiangular tight frame if and only if $D$ is a $(n,k,\lambda)$-difference set in $G$.
\end{thm}
There are related constructions of Grassmannian frames which are not equiangular tight frames using so-called relative difference sets or augmenting equiangular tight frames constructed via Theorem~\ref{thm:DS1} from difference sets with particular parameters~\cite{GoRo09,BH15}.

\section{Equichordal Tight Fusion Frames via Semiregular Divisible Difference Sets}\label{sec:maincon}
The new construction in this section of equichordal tight fusion frames uses semiregular divisible difference sets~\cite{Jung82}.
\begin{defn}
Let $G$ be a finite abelian group of size $mn$ and $D \subseteq G$ a subset of size $k$.  Further let $N \leq G$ be a subgroup of size $n$.  Finally consider the multiset $\Delta(G) = \{d_i - d_j: d_i, d_j \in D, d_i \neq d_j \}$. We say that $D$ is an \emph{$(m,n,k,\lambda_1,\lambda_2)$-divisible difference set} if $\Delta$ contains each element of $(G \cap N^c)$ $\lambda_2$ times and each element of $(N \cap \{0\}^c)$ $\lambda_1$ times.  If further $k > \lambda_1$ and $k^2 - \lambda_2 mn = 0$, we call $D$ \emph{semiregular}.
\end{defn}
The difference sets defined in Definition~\ref{defn:ds} are $(1,n,k, \lambda, \cdot)$-divisible difference sets.  ($\lambda_2$ is superfluous since $G = N$.)
By a simple counting argument, we see that for an $(m,n,k,\lambda_1,\lambda_2)$-divisible difference set,
\beq \label{eqn:ddscount}
k(k-1) = \lambda_1 (n-1) + \lambda_2 (mn-n).
\eeq
We will characterize divisible difference sets using sums of character evaluations \cite{Dav98}.
\begin{lem}\label{lem:div}
$D$ is a $(m,n,k,\lambda_1,\lambda_2)$-divisible difference set in $G$ relative to $N$ if and only if
\[
\abs{\chi(D)}^2=  \left\{ \begin{array}{lr} k^2 & \textrm{if $\chi = \chi_0$} \\ k- \lambda_1  & \textrm{if there is an $h \in N$ with $\chi(h) \neq 1$}  \\ k^2- \lambda_2mn  & \textrm{if $\chi \neq \chi_0$ and $\chi(h)=1$ for all $h \in N$}   \end{array} \right.
\]
\end{lem}
We will make also make use of the following duality result (see, e.g.,~\cite{Pott}).
\begin{lem}\label{lem:ann}
For any subgroup $N \leq G$,
\[
N^\perp = \{\chi \in \widehat{G}: \chi(h) = 1 \textrm{ for all } h \in N \}  
\]
is a subgroup of $\widehat{G}$ (the \emph{annihilator}, alternatively, the characters \emph{principal on $N$}) which is isomorphic to $G/N$.  More precisely, the mapping of $\chi \in N^\perp$ to $\chi' \in \widehat{G/N}$ where 
\[
\chi': G/N \rightarrow S^1 \subset \bC, \quad \chi'(g + N):=\chi(g)
\]
is an isomorphism.
\end{lem}
We now have all of the elements we need to present the construction of Grassmannian fusion frames using divisible difference sets.  The basic idea is to remove the rows corresponding to a semiregular divisible difference set from a character table similar to what was done in Theorem~\ref{thm:DS1} and then cluster the columns according to cosets of the annihilator of $N$ in $\widehat{G}$.  As before, for simplicity in notation, the elements of the difference set are expressed simply as $g$, without any index.
\begin{thm}\label{thm:construct}
Let $D$ be a semiregular $(m,n,k,\lambda_1,\lambda_2)$-divisible difference set in $G$ relative to $N$. Let $\{\eta_j\}_{j=1}^m$ be an enumeration of $N^\perp$, and let $\{\chi_i\}_{i=1}^n$ be a set of coset representatives of $\widehat{G}/N^\perp$. Then for each $1 \leq i \leq n$,
\[
\left\{ e^i_j := \frac{1}{\sqrt{k}}\left((\chi_i \eta_j)(g)\right)_{g \in D} \right\}_{j=1}^m \subset \bF^k
\]
is a set of $m$ flat orthonormal vectors in $\bF^k$.  If we further set for each $1 \leq i \leq n$, $\cW_i = \lspan \{e^i_j\}_{j=1}^m$, then $\{\cW_i\}_{i=1}^n$ is an equichordal tight fusion frame for $\bF^k$ with frame bound $nm/k$ consisting of $n$ $m$-dimensional subspaces.
\end{thm}
\begin{proof}
The vectors are flat by construction.  Since $D$ is semiregular, $k^2 - \lambda_2 mn = 0$.  Plugging this into Lemma~\ref{lem:div} and using Lemma~\ref{lem:ann}, we obtain
\beq\label{eqn:innpro}
\absip{(\chi(g))_{g \in D}}{(\tilde{\chi}(g))_{g \in D}}^2=  \left\{ \begin{array}{lr} k^2 & \textrm{if $\chi = \tilde{\chi}$} \\ k- \lambda_1  & \textrm{if $\chi (\tilde{\chi}^{-1}) \notin N^\perp$}  \\ 0  & \textrm{if $\chi(\tilde{\chi}^{-1}) \in N^\perp \cap \{ \chi_0\}^c$}   \end{array} \right. .
\eeq
Fix $i \in \{1, \hdots, n\}$ and choose $j, \tilde{j} \in \{1, \hdots, m\}$.  Then using Equation~\ref{eqn:innpro} we can compute $\absip{e^i_j}{e^i_{\tilde{j}}}^2 = \delta_{j,\tilde{j}}$ since  $(\chi_i \eta_j) (\chi_i \eta_{\tilde{j}})^{-1} = \eta_j (\eta_{\tilde{j}})^{-1} \in N^\perp$.
Thus for each $i$, $\{ e_j^i\}_j$ is a set of orthonormal vectors.

Since $\left\{\chi_i \eta_j : j \in \{1, \hdots, m\}, i \in \{1, \hdots, n\} \right\}$ is an enumeration of $\hat{G}$. The matrix
\[
L = \left( e_1^1 \vert e_2^1 \vert \cdots \vert e_m^1 \vert e_1^2 \vert \cdots \vert e_m^n \right)
\]
is the result of removing the rows corresponding to $D$ from $X_G$, rescaling the entries by $1/\sqrt{k}$, and permuting the columns.  Hence the rows of $L$ are equal-norm of norm $\sqrt{nm/k}$ and orthogonal.  Thus, $\{\cW_i\}_{i=1}^n$ is a tight fusion frame with bound $nm/k$.

Note that by Equation~\ref{eqn:innpro}, the modulus of the inner product of any $e_j^i$ and $e_{\tilde{j}}^{\tilde{i}}$ with $i \neq \tilde{i}$ is constant, namely $\sqrt{k - \lambda_1}/k$.  We would like to show that the fusion frame is equichordal.  We begin by defining for each $i \in [n]$, $L_i = (e^i_1 e^i_2 \hdots e^i_m)$. We note that for any $i\neq \tilde{i}$
\begin{align*}
\tr(L_i^\ast L_{\tilde{i}} L_{\tilde{i}}^\ast L_i) &= \tr(L_i^\ast L_{\tilde{i}} (L_i^\ast L_{\tilde{i}} )^\ast) = \sum_{j=1}^m \sum_{\tilde{j}=1}^m \absip{e^{\tilde{i}}_{\tilde{j}}}{e^{i}_{{j}}}^2= \frac{m^2 (k-\lambda_1)}{k^2},
\end{align*}
independent of which $i \neq \tilde{i}$ we started with.  Thus it follows from Theorem~\ref{thm:simplex} that $\{\cW_i\}_{i=1}^n$ is an equichordal tight fusion frame.
\end{proof}

Since this article first appeared online as a preprint, Theorem~\ref{thm:construct} has been generalized to a construction of equichordal tight fusion frames using so-called \emph{difference families}~\cite{FMW21}, using ideas from this paper,~\cite{FiSh20},~\cite{FiSh19}, and~\cite{FJKM17}.

\begin{ex}
In~\cite{Ion00}, a construction of semiregular divisble difference sets is presented that generalizes a construction of so-called relative difference sets.  Namely, if $R$ is a $(m,n,k,\lambda)$-semiregular relative difference set  (that is, a $(m,n,k,0,\lambda)$-semiregular divisible difference set) in a group $G$ relative to a subgroup $N$ and $D$ is an $(n,\ell,\mu)$ difference set in $N$, then $DR$ (that is, the set of all pairwise products) is an $(m,n,k\ell,k\mu,\lambda \ell^2)$-semiregular divisble difference set in $G$ relative to $N$.  Let $q$ be an odd prime power equal to $3$ modulo $4$. This construction applied to Paley difference sets $(q, (q-1)/2,(q-3)/4)$ and a certain class of semiregular relative difference sets $(q,q,q,1)$ found in~\cite{Pott} yields a $(q,q,q(q-1)/2,q(q-3)/4,(q-1)^2/4)$-semiregular divisible difference set that may be explicitly defined in terms of finite fields.  Let $\alpha$ be such that the finite field $\operatorname{GF}(q^2)$ is $\operatorname{GF}(q)$ adjoined with $\alpha$.  Then the constructed semiregular divisible difference set is 
\[
\{ x^2 + y^2 + x\alpha : x, y \in \operatorname{GF}(q), y\neq 0\} \subset \operatorname{GF}(q^2).
\]
The fusion frames generated via Theorem~\ref{thm:construct} applied to such semiregular divisible difference sets are always equiisoclinic.  One can see this by noting that the Paley difference sets generate equiangular tight frames (Theorem~\ref{thm:DS1}) and the $(q,q,q,1)$-relative difference sets generate mutually unbiased bases~\cite{GoRo09}.  The construction of the semiregular divisible difference set corresponds to a Kronecker product on the vector side, resulting in equiisoclinic fusion frames (see, e.g.,~\cite{King19b}).
\end{ex}

\begin{ex}
Another example comes from the construction algorithm in~\cite[Section 2]{Dav98}.  In particular, we let $G = \bZ_2 \times \bZ_6$.  Then $X_G$ is the Kronecker tensor product of the (appropriately scaled) $2\times 2$ and $6 \times 6$ discrete Fourier transform matrices (see Example~\ref{ex:dft}). For $N = \{(0,0),(1,0),(0,3),(1,3)\}$ and $N^\perp = \{(0,0),(0,2),(0,4)\}$, 
\[
D = \{(0,0), (0,1), (0,2), (0,4), (1,0), (1,5)\}
\]
is a $(3,4,6,2,3)$-semiregular divisible difference set. The resulting Grassmannian fusion frame is equichordal  but not equiisoclinic.
\end{ex}
Inspired by~\cite{JMF13}, we could like to find a set of sparse vectors which yield orthonormal bases for the Grassmannian fusion frames constructed via Theorem~\ref{thm:construct}. To do this, we make note of the  following well-known orthogonality relations which follow from Lemmas~\ref{lem:orthrel} and~\ref{lem:ann}.
\begin{lem}\label{lem:subsum}
For a subgroup $N \leq G$, let $\{\eta_j\}_{j=1}^m$ be an enumeration of $N^\perp$. Then
\[
\sum_{j=1}^m \eta_j(g) = \left\{ \begin{array}{lr} m & \textrm{if $g \in N$} \\0  & \textrm{if $g \notin N$}  \end{array} \right.
\]
Further, let $\{h_\ell\}_{\ell=1}^m$ be a set of coset representatives of $G/N$.  Then for any $\eta \in N^\perp$,
\[
\sum_{\ell=1}^m \eta(h_\ell) = \left\{ \begin{array}{lr} m & \textrm{if $\eta = \chi_0$} \\0  & \textrm{if $\eta \neq \chi_0$}  \end{array} \right.
\]
\end{lem}
\begin{thm}\label{thm:sparse}
Let $\{\cW_i = \lspan\{e_j^i\}_{j=1}^m\}_{i=1}^n \in \FF(\bF^k,m,n)$ be equichordal as constructed in Theorem~\ref{thm:construct}, with all other notation the same.  Further let $\{h_\ell\}_{\ell=1}^m$ be a set of coset representatives of $G/N$ and define
\[
U = \frac{1}{\sqrt{m}}\left( \begin{array}{cccc} \eta_1(h_1) & \eta_2(h_1) & \cdots & \eta_m(h_1) \\  \eta_1(h_2) & \eta_2(h_2) & \cdots & \eta_m(h_2) \\ \vdots & \vdots& \ddots & \vdots \\  \eta_1(h_m) & \eta_2(h_m) & \cdots & \eta_m(h_m)  \end{array} \right).
\]
For each $i = 1, \hdots, n$ define
\[
\tilde{L}_i := \left(\begin{array}{c|c|c|c} \tilde{e}^i_1 & \tilde{e}^i_2 & \cdots&  \tilde{e}^i_m\end{array}\right):= \left(\begin{array}{c|c|c|c} e^i_1 & e^i_2 & \cdots&  e^i_m\end{array}\right) U^\ast.
\]
Then for all $i \in 1, \hdots, n$, $\cW_i = \lspan\{\tilde{e}_j^i\}_{j=1}^m$ and the vectors $\{\tilde{e}_j^i\}_{j=1,i=1}^{m,n}$ are almost flat and each vector has support size $k/m$.
\end{thm}
\begin{proof}
We first note that Lemma~\ref{lem:subsum} implies that $U$ is unitary.  Thus $\tilde{L}_i\tilde{L}_i^\ast = L_i L_i^\ast$ for each $i$ and  $\{\tilde{e}_j^i\}$ yields the same fusion frame as $\{e_j^i\}$ in Theorem~\ref{thm:construct}.

We now would like to characterize the entries of the $\{\tilde{e}_j^i\}$.  Fix $g \in D$, $\ell \in 1, \hdots, m$, and $i \in 1, \hdots, n$.  Then we may apply Lemma~\ref{lem:subsum} to obtain
\begin{align*}
\tilde{e}_\ell^i (g) &= \frac{1}{\sqrt{km}} \sum_{j=1}^m (\chi_i \eta_j)(g) \overline{\eta_j(h_\ell)} = \frac{1}{\sqrt{km}} \chi_i(g) \sum_{j=1}^m \eta_j(g(h_\ell)^{-1}) \\
&=   \left\{ \begin{array}{lr} \frac{\sqrt{m}}{\sqrt{k}} \chi_i(g) & \textrm{if $g(h_\ell)^{-1} \in N$} \\0  & \textrm{if $g(h_\ell)^{-1} \notin N$}  \end{array} \right.
\end{align*}

Since each $\tilde{e}_j^i$ is unit-norm and $\chi_i(g)$ is always unimodular, we can characterize how many entries of any given vector are nonzero, namely $k/m$. Note that this means that $k/m \in \bN$ for semiregular divisible difference sets.
\end{proof}
Note that for a fixed $j \in 1, \hdots, m$, the support of each $\tilde{e}_j^i$ is the same.  The difference is the component-wise modulation by $\chi_i(g)$.\\

There are two papers in the literature which appeared before this article and in some sense use difference sets to construct equichordal tight fusion frames \cite{Cre08,BoPa15}.  The construction in the former paper involves difference sets of non-abelian groups, like the symplectic groups. The subspaces for the Grassmannian packing are created under orbits of the non-abelian groups.  Thus in that construction, one does not have direct control over whether the bases or projections associated to the subspaces are sparse or flat.  The latter construction starts with an $(n,k,\lambda)$-difference set $D$ in an abelian group.  Then instead of restricting $X_G$ to the rows corresponding to $D$, the rows outside of $D$ are zeroed out, resulting in $n$ vectors in $\bF^n$ which span a $k$-dimensional subspace.  The elements of $D$ are all shifted by an element in $G$ to obtain a new difference set and new set of $n$ vectors in $\bF^n$ which span a $k$-dimensional subspace.  This process is repeated for all elements of $G$, resulting in $n$ different subspaces.  One can find very spare orthonormal bases for the subspaces, namely, the $k$ standard orthonormal basis vectors representing the support set $D-g$ for each $g \in G$. Since divisible difference sets can sometimes be constructed using standard difference sets \cite{Pott}, there is the question of whether the construction in Theorem~\ref{thm:construct} is ever equivalent to the construction in \cite{BoPa15}.    Since the construction in \cite{BoPa15} results in $n$ subspaces of $\bF^n$, we would need a semiregular divisible difference set with $n = k$ for Theorem~\ref{thm:construct} to yield the same fusion frame.  We plug $n = k$ and the semiregular condition $k^2 - \lambda_2 mn =0$ into Equation~\ref{eqn:ddscount} to obtain $k(\lambda_1 + 1- \lambda_2 )= \lambda_1$.  Since $k > \lambda_1$, this can only happen if $\lambda_1 + 1- \lambda_2 = 0 = \lambda_1$, which in turn implies that the original difference set was trivially the entire group. We also note that so-called \emph{paired difference sets} are leveraged in~\cite{FIJKM20} to generate Grassmannian fusion frames that the not equivalent to ones constructed in Theorem~\ref{thm:construct}.

\section{Integrality Conditions}\label{sec:intc}

The goal of this section is to prove certain integrality conditions that must be satisfied for a Grassmannian fusion frame $\{\cW_i\}$ to be flat or almost flat. The results are generalizations of the conditions in \cite{STDH07} for flat equiangular tight frames to exist and similarly use Theorem~\ref{thm:algint}, which concerns algebraic integers, as the main mathematical tool in the proofs. For further reading on algebraic integers, see a standard algebra text like \cite{DummF}. 

\begin{defn} The \emph{algebraic integers} are the roots of monic polynomials in $\bZ[x]$.
\end{defn}

\begin{thm}\label{thm:algint} 
The intersection of the algebraic integers with the rationals is the integers.
\end{thm}
As a simple example of this fact, we note that the square root of a positive integer is either an integer or irrational.
\begin{thm}\label{thm:integral1}
Let $\mathcal{A}$ be a subring of the algebraic integers which is closed under conjugation.  Further let $\{\cW_i\}_{i=1}^n \in \FF(\bF^k,m,n)$ be equichordal.  For each $i \in [n]$, fix an orthonormal basis of $\cW_i$ and set the basis elements to be the columns of the matrix $L_i$. If for all $i \in [n]$ the entries of $\sqrt{k} L_i$ are in $\cA$, then
\[
\frac{km(mn-k)}{n-1} \in \bZ.
\]
\end{thm}
We note that when $m=1$, this yields the same result about flat equiangular tight frames as \cite[Theorem 18]{STDH07}.
\begin{proof}
By applying Definitions~\ref{defn:chor} and~\ref{defn:diffeq} for $i \neq j$
 and using the fact that $\cA$ is a ring closed under conjugation, we obtain
\begin{align*}
\cA &\ni \tr(k^2 L_i^\ast L_j L_j^\ast L_i) = k^2 \tr(P_i P_j) \\
&= \frac{k^2 m(mn-k)}{k(n-1)} = \frac{k m(mn-k)}{n-1}.
\end{align*}
As $\cA \cap \bQ = \cA \cap \bZ$ (Theorem~\ref{thm:algint}), it follows that $\frac{k m(mn-k)}{n-1} \in \bZ$.
\end{proof}
Theorem~\ref{thm:integral1} in particular holds if $\cA = \bZ[\zeta]$ for $\zeta$ a primitive root of unity. We note that Theorem~\ref{thm:integral1} applied to $\bZ[\zeta]$ is weaker than the corresponding result, [Corollary 19] in \cite{STDH07}, due to the fact that the equichordal condition, even when $m=1$, is squared relative to the equiangular tight frame condition.

\begin{thm}\label{thm:inteqio}
Let $\mathcal{A}$ be a subring of the algebraic integers which is closed under conjugation.  Further let $\{\cW_i\}_{i=1}^n \in \FF(\bF^k,m,n)$ be equiisoclinic.  For each $i \in [n]$, fix an orthonormal basis of $\cW_i$ and set the basis elements to be the columns of the matrix $L_i$. If for all $i \in [n]$ the entries of $\sqrt{k} L_i$ are in $\cA$, then
\[
\frac{k(mn-k)}{n-1} \in \bZ.
\]
\end{thm}
\begin{proof}
Since the fusion frame is equiisoclinic,~\eqref{eqn:equiiso} yields that $\alpha = \frac{mn-k}{k(n-1)}$, and it follows from Definition~\ref{defn:diffeq} that $k^2 L_i^\ast L_j L_j^\ast L_i = k^2 \alpha I_k$ has entries in $\cA$.  Applying Theorem~\ref{thm:algint}, we obtain $k^2 \alpha = \frac{k(mn-k)}{n-1}  \in \cA \cap \bQ \subset \bZ$.
\end{proof}

Another integrality condition is as follows.
\begin{thm}\label{thm:int2a}
Let $\{\cW_i\}_{i=1}^n \in \FF(\bF^k,m,n)$.  For each $i \in [n]$, fix an orthonormal basis of $\cW_i$ and set the basis elements to be the columns of the matrix $L_i$. If for all $i \in [n]$ the entries of $\sqrt{k} L_i$ are all $q$th roots of unity with $q = p^s$, $p$ prime, then the following hold.
\begin{itemize}
\item \cite[Theorem 20]{STDH07} $p$ divides $nm$;
\item \cite{Massey1993,KaPa03} If $q=2$, then $4$ divides $nm$, $k=nm=2$, or trivially $k=nm=1$.
\end{itemize}
\end{thm}
The statement of [Theorem 20] in \cite{STDH07} includes the hypothesis that the considered system is an equiangular tight frame; however, that restriction is not used at all in the proof.  One can always form a flat, real tight frame (i.e., $q=2$) by removing rows of a Hadamard matrix and appropriately scaling. An $n\times n$ Hadamard matrix only exists if $n =2$ or $4$ divides $n$ (see, e.g., \cite{HadMat}).  Theorem~\ref{thm:int2a} tells us that we can only form real, flat tight frames with dimensions that suggest they could have come from a Hadamard matrix.\\
\begin{conj}
Real, flat tight frames must come from an appropriately scaled submatrix of a Hadamard matrix.
\end{conj}
In the non-redundant case of real equal-norm flat orthonormal bases, the conjecture is trivially true as that is the definition of a Hadamard matrix.

\section*{Acknowledgements}
The author is indebted to Bernhard Bodmann, Matt Fickus, Joey Iverson, John Jasper, and Dustin Mixon, and  for interesting discussions on the connections between combinatorial design theory and frames.  In particular, Dustin Mixon gave an enlightening talk on tight fusion frames during the Summer of Frame Theory held at the Air Force Institute of Technology. The author was supported during the initial stages of this research in part by the Explorationsprojekt ``Hilbert Space Frames and Algebraic Geometry'' funded by the Zentrum f\"ur Forschungsf\"orderung der Universit\"at Bremen.


\newcommand{\etalchar}[1]{$^{#1}$}
\providecommand{\bysame}{\leavevmode\hbox to3em{\hrulefill}\thinspace}
\providecommand{\MR}{\relax\ifhmode\unskip\space\fi MR }
\providecommand{\MRhref}[2]{%
  \href{http://www.ams.org/mathscinet-getitem?mr=#1}{#2}
}
\providecommand{\href}[2]{#2}

\end{document}